\documentclass[12pt]{amsart}

\usepackage{amsfonts}

\usepackage[alphabetic]{amsrefs}
\usepackage{amscd}
\usepackage{amsmath}
\usepackage{amssymb}
\usepackage{amsthm}
\usepackage{bigdelim}
\usepackage{color}
\usepackage{enumerate}
\usepackage{mathrsfs}
\usepackage{multirow}
\usepackage[all]{xy}
\usepackage{ulem}

\renewcommand{\color}[1]{}

\def\Q{{\mathbb{Q}}}

\def\Pic{{\mathrm{Pic}}}

\theoremstyle{plain}
\newtheorem{thm}{Theorem}[section]
\newtheorem{cor}[thm]{Corollary}
\newtheorem{prop}[thm]{Proposition}

\newtheorem{lem}[thm]{Lemma}
\theoremstyle{definition}

\theoremstyle{remark}
\newtheorem{rem}[thm]{Remark}

\newtheorem{notation}[thm]{Notation}

\begin{document}
	
	\title{Good minimal models with nowhere vanishing holomorphic $1$-forms}
	
	\author {Feng Hao}
	\address{School of Mathematics, Shandong University, Jinan 250100, P.R.China,}
	\email {feng.hao@sdu.edu.cn}
	
	\author {Zichang Wang}
	\address{School of Mathematical Science, University of Science and Technology of China, Hefei 230026, P.R.China.}
	\email {wangzichang@mail.ustc.edu.cn}
	
	\author {Lei Zhang}
	\address{School of Mathematical Science, University of Science and Technology of China, Hefei 230026, P.R.China.}
	\email {zhlei18@ustc.edu.cn}
	
	\maketitle
	
	\begin{abstract}
		Popa and Schnell show that any holomorphic 1-form on a smooth projective variety of general type has zeros. In this article, we show that a smooth good minimal model has a holomorphic 1-form without zero if and only if it admits an analytic fiber bundle structure over a positive dimensional abelian variety.
	\end{abstract}
	
	\bigskip
	
	\section{Introduction}
	
	Holomorphic 1-forms as core objects in the study of algebraic geometry encode much geometric information of irregular varieties. As their dual objects, holomorphic vector fields were intensely studied by Baum, Bott, Carrell, Howard, Lieberman, Matsushima, etc. (see e.g. \cite{Bot67, Mat69, BaBo72, How72, CaLi73, Car74}) around 1970. In particular, a recent result by  Amor\'os, Manjar\'in and Nicolau \cite{Amorós2012} shows that the existence of nowhere vanishing holomorphic vector fields gives a strong structural information of a compact K\"ahler manifold.
	
	We also expect that the existence of nowhere vanishing holomorphic 1-forms restricts varieties a lot. In fact, a celebrated result of Popa and Schnell \cite{popa2014kodaira} shows that any holomorphic 1-form on a smooth complex projective variety of general type has zeros (See also \cite{zhang1997global,hacon2005holomorphic,luo2005holomorphic} for pioneer works on this result).
	
	A fine research on smooth projective varieties admitting holomorphic 1-forms without zero was initiated by Schreieder \cite{schreieder2021zeros}, in which Schreieder studies how holomorphic 1-forms without zeros affect the topology of the compact K\"ahler manifolds, and classifies smooth projective surfaces which admit holomorphic 1-forms without zeros (see also \cite{hao2021equality,hao2021holomorphic,kotschick2022holomorphic,cch25nowhere,hao2024nowhere} for related results and higher dimensional generalizations). We would like to mention that a recent related result by Catanese \cite{catanese2024manifoldstrivialchernclasses} gives a much finer study on the compact K\"ahler manifolds admitting nowhere vanishing holomorphic 1-forms arising from coframed cotangent bundles. In this article, we give a thorough study on the  classification of good minimal models admitting nowhere vanishing holomorphic 1-forms. When this article \textcolor{red}{was} in preparation, we notice that Church has claimed similar results using \textcolor{red}{a} different approach in his article \cite{church2024vanishing1formsvarietiesadmitting} very recently.
	
	In this paper, we make the following conventions.
	 \textcolor{red}{By a \textit{variety} we mean an integral separated scheme of finite type over an algebraically closed field, and by a \textit{fibration}, we mean a projective morphism $f:X\to Y$ of normal quasi-projective varieties such that $f_{*}\mathcal{O}_{X}=\mathcal{O}_{Y}$.  We work over the complex number field $\mathbb{C}$ if not specially mentioned.}
    
    Our main result is the following ``Beauville-Bogomolov'' type theorem for a fibration, which induces our classification of good minimal models admitting nowhere vanishing holomorphic 1-forms. Moreover, it also induces \cite[Theorem A]{church2024vanishing1formsvarietiesadmitting}.
	
		\begin{thm}\label{thm:structure}
		Let $f\colon X\to S$ be a fibration of normal projective varieties and $h\colon X\to B$ be a morphism to an abelian variety.
		$$\xymatrix{&X\ar[r]^{h}\ar[d]^{f}&B\\
			&S&}$$
		Suppose that
		\begin{itemize}
			\item[(i)] $X$ has at most $\Q$-factorial canonical singularities, and the canonical divisor $K_X$ is semi-ample; %\sout{a good minimal model, that is, it has at most canonical singularities and the canonical divisor $K_X$ is semi-ample;}
			\item[(ii)] for a general fiber $F$ of $f$, $K_F$ is $\mathbb{Q}$-linearly equivalent to $0$; and
			\item[(iii)] the restriction $h|_F\colon F\to B$ is surjective.
		\end{itemize}
		Denote by $g\colon X\to A$ the fibration arising from the Stein factorization of $h\colon X\to B$.
		Then
		\begin{itemize}
			\item [(1)] $A$ is an abelian variety;
			% \sout{ \item[(2)] $g\colon X\to A$ is a fiber bundle; and}
			\item[(2)] there exists an isogeny $A' \to A$ of abelian varieties such that $X\times_A A' \cong Z' \times A',$ where $Z'$ is a fiber of $g$.
		\end{itemize}
		%\sout{As a result, the natural action of $H=ker(\tau)$ on $A'$ induces an action on the base change $X\times_A A'$. Under the identification $\eta\colon X\times_A A' \cong G \times A'$, $H$ acts on $ G \times A'$ diagonally. Moreover we may choose the isogeny $\tau\colon A' \to A$ such that the action of $H$ on $G$ is faithful.}
	\end{thm}

This theorem together with the results in \cite{popa2014kodaira} imply the following classification results for good minimal models with nonwhere vanishing holomorphic 1-forms.
	
	\begin{thm}\label{thm:main2}
		Let $X$ be a smooth good minimal model. Then the followings are equivalent
		
		(i) $X$ admits holomorphic 1-forms $\omega_1, \ldots , \omega_g \in H^0(X, \Omega^1_X)$ which are linearly independent pointwisely, in particular, their zero loci are empty.

		(ii) $X$ admits an isotrivial smooth morphism to an abelian variety $A$ of dimension $\geq g$. More precisely, there is a finite \'etale cover $A'\to A$, such that $X\times_A A'$ is isomorphic to a product $Y\times A'$.

In particular, if $X$ has a holomorphic 1-form $\omega \in H^{0}(X,\Omega^{1}_{X})$ such that the zero locus $Z(\omega)=\emptyset$, then  $X$ admits an isotrivial smooth morphism to a positive dimensional abelian variety $A$.
	\end{thm}	
   \textcolor{red}{In general, if we do not assume $X$ is minimal, the existence of nowhere vanishing holomorphic 1-forms also has a strong restriction on the birational structure of $X$. 
	\begin{thm}\label{thm:main3}
    Let $X$ be a smooth projective variety with 
      $g$ pointwise linearly independent holomorphic 1- forms $\omega_1, \ldots , \omega_g \in H^0(X, \Omega^1_X)$. Suppose that $X$ admits a good minimal model, then there is a finite \'etale cover $X'\to X$ such that $X'$ is birational to a product $Y\times A'$, where $A'$ is an abelian variety of dimension $\geq g$.
	  \end{thm}}
	\textcolor{red}{In general, for a smooth projective variety $X$ of non-negative Kodaira dimension}, the existence of holomorphic 1-forms without zero on $X$ does not necessarily imply that $X$ admits a smooth morphism to a positive dimensional abelian variety. A specific example is given in \cite{Schreieder2022ZerosOO} by Schreieder and Yang: {\color{red}The blow-up of $E_1\times E_2\times Y$ along the union of two curves $E_1\times \{0\}\times \{x\}$ and $\{0\}\times E_2 \times \{y\}$, where $E_1, E_2$ denote non-isogeneous elliptic curves, $Y$ is a simply connected smooth projective variety of non-negative Kodaira dimension, and $x$,$y$ are two distinct closed points of $Y$.}

\medskip

\textbf{The idea of the proof of Theorem \ref{thm:structure}.} First  we have a ``Beauville-Bogomolov'' type decomposition for a general fiber $X_s$ of $f$:  \textit{the Stein factorization of the morphism $X_s \to B$ gives rise to an \'{e}tale fiber bundle $X_s \to Y_s$ over an abelian variety $($\textcolor{red}{see Theorem \ref{thm:pol-isotr}}$)$, that is, there exists an \'{e}tale base change $Y'_s \to Y_s$ such that $X_s \times_{Y_s} Y'_s \cong Z_s \times Y'_s$.}
Since each $Y_s$ is isogenous to $B$, we see that $Y_s$ are isomorphic to each other. Assume general fibers $Y_s \cong A$. It is reasonable to expect that there exists an \'{e}tale base change $A' \to A$, which is independent of $s\in S$, such that the base change of each general fiber $X_s\times_A A'$ splits, in turn $X\times_AA'$ splits (birationally). However, be cautious that we do not have a canonical splitting $X_s\times_{Y_s} Y_s' \cong Z_s\times Y_s'$, that is, the isomorphism is not unique. In fact a splitting $X_s\times_{Y_s} Y_s' \cong Z_s\times Y_s'$ is determined by a polarization on $X_s$ (see Proposition \ref{prop:pol-isotr}). To implement this strategy, we shall fix a polarization on $X$ and apply the Isom functors of polarized varieties.

\medskip
	
	{\small \noindent\textit{Acknowledgments.}
		We would like to thank Stefan Schreieder for very helpful comments. We are also grateful to the referee for his/her careful readings, and very helpful suggestions and corrections to improve this manuscript. This research is partially supported by CAS Project for Young Scientists in Basic Research, Grant No. YSBR-032.	The first author is supported by NSFC No.1240010723, SDNSFC (No. 2024HWYQ-009, No. ZR2024MA007, No. tsqn202312060). The third author is partially supported NSFC (No.12122116, 12471495).
	%\textbf{Conventions and Notations.}

	\section{Locally trivial fibrations}
	In this section we treat  \textit{locally trivial} fibrations. \textcolor{red}{Let $f\colon X\to T$ be a fibration of normal varieties over an algebraically closed field. We say $f\colon X\to T$ is \textit{locally trivial} if  the following condition is satisfied
\begin{itemize}
\item for each closed point $t\in T$, denoting by $\hat{t}=\mathrm{Spf}(\widehat{\mathcal{O}}_{T,t})$ the formal spectrum of the local ring $(\mathcal{O}_{T,t},m_{t})$, $X\times_{T} \hat{t}$ is isomorphic to $X_t \times \hat{t}$ as formal schemes over $\hat{t}$.
	\end{itemize}} \textcolor{red}{For a smooth fibration $f:X\to T$ over $\mathbb{C}$, by \cite[Theorem 5.1]{Kodaira-spencer-I} the above condition is equivalent to the following one
	\begin{itemize}
		\item
		for each closed point $t\in T$, there exists an analytic neighborhood $U$ such that $X_{U}$ is analytically isomorphic to $X_{t}\times U$. 
	\end{itemize}}
We shall prove that for a locally trivial fibration  $f\colon X\to T$, under the condition that $K_{X_t}\equiv 0$ (\textcolor{red}{here the symbol ``$\equiv$`` denotes the numerical equivalence of $\mathbb{Q}$-Cartier divisors, see \cite[Notation 0.4]{KM98bir}}), there exists an \'etale $G$-cover $T'\to T$ such that
	$X \cong (X_{t_0}\times T')/G$, where $X_{t_0}$ is a closed fiber of $f$, $G\leq \mathrm{Aut}(X_{t_0})$ acts on $X_{t_0}\times T'$ diagonally.
	In the following we shall denote this structure by $X \cong X_{t_0}\times^G T'$.

\subsection{Isom Functors}

In this subsection we review the contruction and some basic properties of Isom functors from \cite[Section 7]{patakfalvi2019beauville}, which is the main technical tool in this article. {\color{red}In this section, we work over an algebraically closed field $k$.}
	
	Let $T$ be a normal noetherian scheme. Let $f^{i}\colon X^i\to T$ ($i = 1, 2$) be two flat families of geometrically normal projective varieties over $T$. Let $L_{i}$ be $f^{i}$-ample line bundles over $X^{i}$.
	
\textcolor{red}{First consider the following  functor
    \begin{itemize}
        \item[(1)] $\mathbf{Isom_{T}^{pre}}((X^{1},L_{1}),(X^{2},L_{2}))(-)\colon\mathbf{(Sch/T)^{op}}\longrightarrow \mathbf{Set}$ which maps
        $$S \longmapsto \left\{
   \alpha\colon X^1_S\cong_{S} X^2_S 
   \;\middle|\;
   \begin{array}{c c c}
        \chi (((L_{1})_{S}\otimes \alpha^{*}(L_{2})_{S})^{\otimes m}_{s})=\chi(((L_{1})_{S}^{\otimes 2m})_{s}),\\
      \forall s\in S,~\forall m\in \mathbb{Z}
   \end{array}
\right\}.$$
\end{itemize} 
Here $X^{i}_{S}$ denotes $X^i\times_{T}S$, and $p_{i}\colon X^{i}\times_{T}S\to X^{i}$ is the natural projection onto the first factor, and we use $(L_{i})_S$ to denote $p_{i}^{*}L_{i}$. There is a natural map from this functor to the Picard functor
$$\mathbf{pol} :\mathbf{Isom_{T}^{pre}}((X^{1},L_{1}),(X^{2},L_{2}))\to \mathbf{Pic}(X^{1}/T)$$ given by the following assignment for every $S\in \mathbf{(Sch/T)}$
\begin{align*}
    \mathbf{pol}(S):\mathbf{Isom_{T}^{pre}}((X^{1},L_{1}),(X^{2},L_{2}))(S)&\longrightarrow \mathbf{Pic}(X^{1}/T)(S)=
    {\Pic}(X^{1}_S/S)\\
    \alpha&\longmapsto (L_{1})_S\otimes\alpha^{*}(L_{2}^{-1})_S.
\end{align*}
It is known that the Picard functor $\mathbf{Pic(X^{1}/T)}$ is represented by a flat projective group scheme $\Pic(X^{1}/T)$ over $T$, see \cite[Chapter 9]{books2005fundamental}.
Next we introduce the following two functors
    \begin{itemize}
\item[(2)] $\mathbf{Isom^{pre}_{T,0}}((X^1,L_{1}),(X^2,L_{2}))(-):\mathbf{(Sch/T)^{op}}\to\mathbf{Set}$ which maps 
$$S\mapsto\{\alpha\colon X^1_S\cong_{S} X^2_S ~|~(L_{1})_S\otimes\alpha^{*}(L_{2})_S^{-1}\in\Pic^{0}(X^{1}_S/S)\};$$
\item[(3)] $\mathbf{Isom_{T}}((X^1,L_{1}),(X^2,L_{2}))(-)
		\colon \mathbf{(Sch/T)^{op}}\longrightarrow \mathbf{Set}$ which maps 
        $$S\longmapsto\{\alpha\colon X^1_S\cong_{S} X^2_S~|~(L_{1})_S\cong_{S} \alpha^{*}(L_{2})_S \}.$$
        Here $(L_{1})_S\cong_{S} \alpha^{*}(L_{2})_S$ means $(L_{1})_S\otimes\alpha^{*}(L_{2})_S^{-1} \in \Pic(S)$.
\end{itemize} }
\smallskip

\textcolor{red}{By \cite[Construction 7.5]{patakfalvi2019beauville} we have}

\begin{prop}\label{prop:isom-rep} With the above notation.

\textcolor{red}{(i) The three functors (1,2,3) are represented by the following quasi-projective schemes over $T$ respectively
\begin{small}
    $$\mathrm{Isom_{T}^{pre}}((X^{1},L_{1}),(X^{2},L_{2})),~\mathrm{Isom_{T,0}^{pre}}((X^{1},L_{1}),(X^{2},L_{2}))~ and ~\mathrm{Isom}_{T}((X^1,L_{1}),(X^2,L_{2}))$$
\end{small}}
\textcolor{red}{which are compatible with base change, say, for a base change $S\to T$, we have
		$$\mathrm{Isom}_{T}((X^1,L_{1}),(X^2, L_{2}))\times_{T}S\cong \mathrm{Isom}_{S}((X^1_{S}, (L_{1})_{S}),(X^2_{S}, (L_{2})_{S})).$$}

\textcolor{red}{(ii) The map $\mathbf{pol} :\mathbf{Isom_{T}^{pre}}((X^{1},L_{1}),(X^{2},L_{2}))\to \mathbf{Pic}(X^{1}/T)$ corresponds to the morphism over $T$
    $$\mathrm{pol} :\mathrm{Isom_{T}^{pre}}((X^{1},L_{1}),(X^{2},L_{2}))\to \Pic(X^{1}/T).$$}

\textcolor{red}{(iii) We can fit the above schemes into a commutative diagram of the base changes
    \begin{small}
        $$\xymatrixcolsep{0.13in}
         \xymatrix{\mathrm{Isom_{T}}((X^{1},L_{1}),(X^{2},L_{2}))\ar[r]\ar[d]  &\mathrm{Isom_{T,0}^{pre}}((X^{1},L_{1}),(X^{2},L_{2}))\ar[r]\ar[d]  &\mathrm{Isom_{T}^{pre}}((X^{1},L_{1}),(X^{2},L_{2}))\ar[d]^{\mathrm{pol}}\\
0_{T}\ar@{^{(}->}[r] &\Pic^0(X^{1}/T)\ar@{^{(}->}[r]  &\Pic(X^{1}/T)}$$
    \end{small}
         where the morphism $\Pic^0(X^{1}/T) \hookrightarrow \Pic(X^{1}/T)$ is the inclusion of the connected component containing the zero section and $0_{T}\hookrightarrow\Pic^0(X^{1}/T)$ is the inclusion of the zero section $0_{T}$.}
	\end{prop}

	\begin{rem}\label{rem:fact-isom}
		The following statements will be used in the sequel.
		\begin{itemize}
			\item [(a)] Let $\mathrm{Aut}_{T}(X^1,L_{1})=\mathrm{Isom}_{T}((X^1,L_{1}),(X^1,L_{1}))$, which is by definition a group scheme over $T$. The two group schemes $\mathrm{Aut}_{T}(X^1,L_{1})$ and $\mathrm{Aut}_{T}(X^2,L_{2})$ act on $\mathrm{Isom}_{T}((X^1,L_{1}),(X^2,L_{2}))$ by composition in a natural way.
Then by Proposition \ref{prop:isom-rep} (i), for a point $t\in T$, the fiber \textcolor{red}{of $\mathrm{Isom}_{T}((X^1,L_{1}),(X^2,L_{2})) \to T$ over $t$ is} $\mathrm{Isom}_{k(t)}((X^1_{t},\textcolor{red}{({L_{1}})_{t}}),(X_t^2,\textcolor{red}{({L_{2}})_{t}}))$. Once $(X^{1}_{t},\textcolor{red}{({L_{1}})_{t}})\cong (X^{2}_{t},\textcolor{red}{({L_{2}})_{t}})$ then \textcolor{red}{the natural action of $\mathrm{Aut}_{k(t)}(X^{1}_{t},({L_{1}})_{t})$ on $I_{t}$ induces an isomorphism $\mathrm{Isom}_{k(t)}((X^1_{t},\textcolor{red}{({L_{1}})_{t}}),(X_t^2,\textcolor{red}{({L_{2}})_{t}}))\cong \mathrm{Aut}_{k(t)}(X^{1}_{t},({L_{1}})_{t})$ as schemes over $k(t)$.}  
               
			\item [(b)] \textcolor{red}{Similarly we can consider the action of 
            \begin{footnotesize}
            $\mathrm{Aut}^0_{T}(X^1,L_{1}):=\mathrm{Isom}_{T,0}^{pre}((X^1,L_{1}),(X^1,L_{1}))$\end{footnotesize} on $I^{pre}_{0}=\mathrm{Isom_{T,0}^{pre}}((X^{1},L_{1}),(X^{2},L_{2}))$.    
             For a point $t\in T$, once there exists an isomorphism $\alpha: X^{1}_{t} \cong_{k(t)} X^{2}_{t}$ such that $(L_{1})_t\otimes \alpha^*(L^{-1}_2)_t\in\Pic^{0}(X^{1}_t)$, then the fiber of $I^{pre}_{0}\to T$ over $t$ is isomorphism to $\mathrm{Aut}^0_{k(t)}(X^1_t,(L_{1})_{t})$ over $k(t)$.}
            \end{itemize}
	\end{rem}

	\subsection{Automorphism of polarized abelian varieties}
	\begin{thm}{\textcolor{red}{\textup{(}\cite{mumford1970abelian}*{II.8 Theorem 1}\textup{)}}}\label{pol-av}
		Let $A$ be an abelian variety over an algebraically closed field $k$, and let $L$ be an ample line bundle over $A$. Then the morphism
		\begin{align}
			\phi_{L}\colon A&\longrightarrow \mathrm{Pic}^{0}(A) \notag \\
			a&\longmapsto t^{*}_{a}L\otimes L^{-1}\notag
		\end{align}
		is an isogeny of abelian varieties.
	\end{thm}
    
	\begin{prop}{\textcolor{red}{\textup{(}\cite{patakfalvi2019beauville}*{Proposition 10.1}\textup{)}}}\label{finite-auto-pol}
		If $(X,\Delta)$ is a projective klt pair over an algebraically closed field $k$, such that $K_{X}+\Delta$ is pseudo-effective and $L$ is an ample line bundle on $X$, then
		$$\mathrm{Aut}((X,\Delta);L):=\{\sigma\in\mathrm{Aut}(X,\Delta)~|~\sigma^{*}L\cong L\}$$
		is finite.
	\end{prop}

	\textcolor{red}{The following lemma is a slight generalization of \cite[IV.21.Theorem 5]{mumford1970abelian}, a result due to Serre.}
	\begin{lem}\label{lem:pol-mor-finite}
		Let $(A',L')$ and $(A,L)$ be two polarized abelian varieties of same dimension \textcolor{red}{over an algebraically closed field $k$}, where $L',L$ are ample line bundles on $A',A$ respectively. Assume that there exists a morphism $\alpha:A'\to A$ such that $\alpha^*L \cong L'$.
		Then the functor $\mathbf{Mor}((A',L'),(A,L))$ defined by \[ \hbox{$k$-scheme}~T \mapsto \{ \alpha\colon A'_T\to A_T \mid \hbox{$\alpha$ is surjective and } \alpha^*(L_T) \cong L'_T\}
		\]
		is represented by a finite scheme over $k$.
	\end{lem}
	\begin{proof}
		According to \cite[Construction~7.5]{patakfalvi2019beauville}, $\mathbf{Mor}((A',L'),(A,L))$ is represented by a quasi-projective scheme \textcolor{red}{over $k$}, denoted by $\mathrm{Mor}((A',L'),(A,L))$. To show it is a finite scheme over $k$, it suffices to show that the following set
		\[    \{ \alpha\colon A'\to A \mid \hbox{ $\alpha$ is surjective, and } \alpha^*L \cong L'\}
		\]
		consists of finitely many elements. Remark that this set is nonempty by the assumption.

		For each surjective morphism $\alpha\colon A'\to A$, we can factor $\alpha$ into an isogeny with a translation as follows:
		$$\alpha:A'\xrightarrow{t_{x}}A'\xrightarrow{i}A$$
		where $t_x\colon A'\to A'$ denotes the translation by $x\in A'$.
		Since $\alpha^{*}L\cong L'$, the degree of $\alpha$ is a fixed number $d=\frac{L'^n}{L^n}$, where $n=\mathrm{dim}A$. Note that there are only finitely many isogenies $A'\to A$ of degree $d$ (\cite{edixhoven2012abelian}*{Proposition~7.14}). It follows that there are only finitely many choices of $i:  A'\to A$.
		
		Fix an isogeny $i:A'\to A$, $i^{*}L$ is an ample line bundle on $A'$. By Theorem \ref{pol-av}, the morphism
		$$\phi_{i^{*}L}:A'\to \mathrm{Pic}^{0}(A')~,~a'\mapsto t_{a'}^{*}(i^{*}L)\otimes(i^{*}L^{-1})$$
		is an isogeny of abelian varieties. As $\phi_{i^*L} (x)= t_x^*(i^*L) \otimes (i^*L^{-1}) \cong L'\otimes (i^{*}L^{-1})$, there are only finitely many $x\in A'$ such that $t_{x}^{*}i^{*}L\cong L'$. In summary, we obtain that there are only finitely many $\alpha\in \mathrm{Mor}(A',A)$ such that $\alpha^*L\cong L'$. And in turn, we conclude that $\mathrm{Mor}((A',L'),(A,L))$ is a finite scheme over $k$.
	\end{proof}

	\subsection{Trivialization of polarized isotrivial fibrations}
	From here to the end of this paper, we work over the field $\mathbb{C}$ of complex numbers.
	\begin{prop}\label{prop:pol-isotr}
		\textcolor{red}{Let $X$, $T$ be two normal quasi-projective varieties.} Let $f\colon X\rightarrow T$ be a locally trivial fibration and $L$ an $f$-ample line bundle on $X$. Fix a closed point $t_0\in T$, denote by $X_0$ the fiber over $t_0$ and set $L_0=L|_{X_0}$.  Assume that {\color{red} $X_0$ is not uniruled and}
		\begin{itemize}
			\item[$\diamond$] for each closed point $t\in T$, there exists a polarized isomorphism $(X_{t},L_{t}:=L|_{X_{t}})\cong(X_{0},L_{0})$.
		\end{itemize}
		Then the following statements hold true.
		
		(1) The morphism $$\pi\colon I=\mathrm{Isom}_{T}((X,L),(X_0\times T,L_{0} \times T)) \to T$$ is a $G=\mathrm{Aut}(X_0,L_0)$-torsor.
		
		(2) \textcolor{red}{There exists a $G$-equivariant polarized isomorphism over $I$ $$(X\times_{T}I,L\times_{T}I)\cong(X_0\times I,L_{0}\times I),$$ where the right hand side is endowed with the diagonal $G$-action, and the left hand side is endowed with the $G$-action on the second factor induced by the fiber product.
		In particular, the restriction of this isomorphism on each fiber over $t\in T$ is the evaluation map $ X_{t}\times I_{t}\to X_{0}\times I_t$ given by $(x,\sigma) \mapsto (\sigma(x),\sigma)$.}

		(3) Denote by $I_0$ the component of $I$ containing $\mathrm{id}_{X_0} \in \pi^{-1}(t_0)$ and denote by $\textcolor{red}{H=\{g\in \mathrm{Aut}(X_0,L_0)~|~g\cdot I_{0}= I_{0}\}} \leq  \mathrm{Aut}(X_0,L_0)$ the subgroup fixing $I_0$. Then $I_0 \to T$ is an $H$-torsor and $[\mathrm{Aut}(X_0,L_0):H] < +\infty$. Moreover, if writing that $\mathrm{Aut}(X_0,L_0)=\coprod_ig_iH$, then $I = \coprod_ig_iI_0$.

		In particular, if $\mathrm{Aut}(X_0,L_0)$ is a finite group, then $I_0 \to T$ is an \'etale $H$-cover and $X\cong X_0\times^{H} I_0$.
		
		%which fits into the following commutative diagram:
		%$$\xymatrix{&X_{0}\times B'\ar@/^1.5pc/[drr]\ar@/_1.5pc/[ddr]_{p_{2}}\ar[dr]^{\cong}&&&\\
			%&&X\times_{B}B'\ar[r]\ar[d]&X \cong (X_0\times B')/H\ar[d]^{f}&\\
			%&&B'\ar[r]^{\tau}&B \cong B'/H&
			%}$$
	\end{prop}
	\begin{proof}
		(1) Note that by Remark \ref{rem:fact-isom} (a), there exists a natural $G$-action on $I$, and by the assumption \textcolor{red}{$\diamond$, $I\to T$ is surjective and }each closed fiber is isomorphic to $G=\mathrm{Aut}(X_0,L_0)$. \textcolor{red}{It remains to show that $I\to T$ is flat. For a closed point $t\in T$,  set $\hat{t}=\mathrm{Spf}(\widehat{\mathcal{O}}_{T,t})$. Recall that        
        $\mathcal{O}_{T,t} \to \widehat{\mathcal{O}}_{T,t}$ is a faithfully flat extension (see \cite[Theorem 56]{Matsumura}). It suffices to show that $I_{\hat{t}}:=I\times_{T}\hat
        {t}\to\hat{t}$ is flat for every closed point $t\in T$. }

\textcolor{red}{As $f:X\to T$ is locally trivial, there is an isomorphism over $\hat{t}$ $$\alpha:X_{\hat{t}}=X\times_{T}\hat{t}\xrightarrow{\sim}X_{t}\times\hat{t},$$ and automatically $L_{\hat{t}} \otimes\alpha^*(L_t\times{\hat{t}})^{-1}\in\mathrm{Pic}^{0}(X_{\hat{t}}/\hat{t})$. By the assumption $\diamond$, there is an isomorphism $\gamma_{t}:X_{t}\xrightarrow{\sim}X_{0}$ such that $\gamma_{t}^{*} L_{0}\cong L_{t}$. Consider the following composition of isomorphisms over $\hat{t}$ $$\sigma_{\hat{t}}:X_{\hat{t}}\xrightarrow[\cong]{\alpha}X_{t}\times\hat{t}\xrightarrow[\cong]{\gamma_{t}\times id_{\hat{t}}}X_{0}\times\hat{t}.$$ 
Then $L_{\hat{t}} \otimes\sigma_{\hat{t}}^*(L_0\times{\hat{t}})^{-1}=L_{\hat{t}} \otimes\alpha^*(L_t\times{\hat{t}})^{-1}\in\mathrm{Pic}^{0}(X_{\hat{t}}/\hat{t})$. Set ${M_{\hat{t}}}=L_{\hat{t}} \otimes\sigma_{\hat{t}}^*(L_0\times{\hat{t}})^{-1}$ and regard it as a section of $\mathrm{Pic}^{0}(X_{\hat{t}}/\hat{t}) \to \hat{t}$.} 
        
        \textcolor{red}{We can identify $\mathrm{Aut}^{0}(X_0\times \hat{t},L_{0} \times \hat{t}) = \mathrm{Aut}^{0}(X_{0},L_{0})\times{\hat{t}}$. Remark that since $X_0$ is assumed not uniruled, we can write that $\mathrm{Aut}^{0}(X_0,L_{0}) = \coprod_{i=1}^{i=r} g_i \cdot\mathrm{A}_{0}$, where $\mathrm{A}_{0}$ is the connected component of $\mathrm{Aut}^{0}(X_0,L_{0})$ containing $\mathrm{id}_{X_0}$ which is an abelian variety (\cite[Theorem 14.1]{book:Ueno1975}).} 
        
        \textcolor{red}{By Remark \ref{rem:fact-isom} (b), the natural action of $ \mathrm{Aut}^{0}(X_{0},L_{0})\times{\hat{t}}$
     on $I^{pre}_{0,\hat{t}}=I^{pre}_{0}\times_{T}\hat{t}$ induces an isomorphism over $\hat{t}$
       \begin{align*}
         \mathrm{Aut}^{0}(X_{0},L_{0})\times{\hat{t}}&\xrightarrow[\cong]{(-)\circ\sigma_{\hat{t}}}\mathrm{Isom}^{pre}_{0,\hat{t}}\bigg(\big(X_{\hat{t}},L_{\hat{t}}\big),\big(X_0\times \hat{t},L_{0} \times \hat{t}\big)\bigg)=I^{pre}_{0,\hat{t}}\\
         \beta &\longmapsto\beta\circ\sigma_{\hat{t}}.
     \end{align*}
    Consider the following composition morphism
     \begin{align*}
         \phi\colon\mathrm{Aut}^{0}(X_{0},L_{0})\times{\hat{t}}&\xrightarrow[\cong]{(-)\circ\sigma_{\hat{t}}}I^{pre}_{0,\hat{t}}\xrightarrow{\mathrm{pol}_{\hat{t}}}\Pic^{0}(X_{\hat{t}}/\hat{t})\xrightarrow[\cong]{(-)\otimes M_{\hat{t}}^{-1}}\Pic^{0}(X_{\hat{t}}/\hat{t})\\
        \beta &\mapsto \beta\circ\sigma_{\hat{t}}\mapsto L_{\hat{t}}\otimes \sigma^{*}_{\hat{t}}\beta^{*}(L_{0}\times \hat{t})^{-1}\mapsto \sigma^{*}_{\hat{t}}[(L_{0}\times \hat{t})\otimes\beta^{*}(L_{0}\times \hat{t})^{-1}]
     \end{align*}
     which sends $\mathrm{id}_{X_0}\times \hat{t} \mapsto 0_{\hat{t}}$.
    By \cite[Corollary 6.4]{MF82}, $\phi|_{\mathrm{A}_{0}\times\hat{t}}: \mathrm{A}_{0} \times \hat{t} \to \Pic^{0}(X_{\hat{t}}/\hat{t})$     
    is a homomorphism of group formal schemes over $\hat{t}$, hence $\mathrm{A}_{0} \times \hat{t}$ is flat over its image, so is $\mathrm{Aut}^{0}(X_{0},L_{0})\times{\hat{t}} \to \Pic^{0}(X_{\hat{t}}/\hat{t})$. 
    It follows that $\mathrm{pol}_{\hat{t}}:I^{pre}_{0,\hat{t}} \to \Pic^{0}(X_{\hat{t}}/\hat{t})$ is flat over its scheme-theoretic image $S_{\hat{t}}$ (in the sense of \cite[II.3 Exercises 3.11(d)]{har1977}), which is a closed subscheme of $\Pic^{0}(X_{\hat{t}}/\hat{t})$. In turn, we conclude that $I^{pre}_{0,\hat{t}} \to S_{\hat{t}}$ is flat, more precisely $I^{pre}_{0,\hat{t}}$ is a $G$-torsor over $S_{\hat{t}}$.}

    \textcolor{red}{We can identify $I_{\hat{t}} \cong (I^{pre}_{0,\hat{t}})_{0_{\hat{t}}}$ as formal schemes over $\hat{t}$. If $I_{\hat{t}}\to \hat{t}$ is not flat over $\hat{t}$, then its scheme-theoretic image, being isomorphic to $0_{\hat{t}} \bigcap S_{\hat{t}} $, is a proper closed subscheme of $\hat{t}$. But this contradicts that $I \to T$ is surjective. In conclusion, $I \to T$ is a $G=\mathrm{Aut}(X_0, L_0)$-torsor.}    

		\medskip

		(2)  Applying Proposition \ref{prop:isom-rep} (i), we have an isomorphism over $I$
		$$I\times_{T}I\cong \mathrm{Isom}_{I}((X\times_{T}I,L\times_{T}I),(X_0\times I,L_{0} \times I)),$$
		which corresponds to an isomorphism between representable functors. 
      {\color{red} Thus the diagonal morphism
      $$\xymatrixcolsep{0.4in}\xymatrix{I\ar[dr]_{\mathrm{id}}\ar[r]^(.35){\Delta_{I}}&I\times_{T}I\ar[d]^{\mathrm{pr}_{2}},\\
      &I}$$
      viewed as a morphism of $I$-schemes, gives rise to a polarized  isomorphism over $I$
      $$\psi:(X\times_{T}I,L\times_{T}I)\xrightarrow{\sim}(X_0\times I,L_{0}\times I).$$
      Next we show that $\psi$ is $G$-equivariant. We verify this pointwisely. Let $(x,\sigma) \in X\times_{T}I$ be a closed point, where $x\in X_t$ and $\sigma \in I_t=\mathrm{Isom}_{k(t)}((X_{t}, L_t), (X_0, L_0))$ for some $t\in T$. By construction, we have $\psi (x,\sigma)=(\sigma(x),\sigma)\in X_{0}\times I$. The action of $G=\mathrm{Aut}(X_0,L_0)$ on $X\times_{T}I$ is induced by the action of $G$ on the second factor $I$, precisely, for $g\in G$, $g\cdot(x,\sigma)=(x,g\cdot\sigma)$. It follows that   
      $$\psi(g\cdot(x,\sigma))=\psi(x,g\cdot\sigma)=(g\cdot\sigma(x),g\cdot\sigma) = g\cdot \psi(x,\sigma).$$
  This shows that $\psi$ is a $G$-equivariant polarized isomorphism over $I$.} 
        
        %we can verify that $\psi$ is $G$-equivariant pointwisely.
        %Since $I\to T$ is a $G$-torsor, by \cite[Proposition 4.43]{books2005fundamental} we have an isomorphism $G\times I \overset{\sim}{\to} I\times_T I$ via $(g, \sigma)\mapsto (g\cdot\sigma, \sigma)$, thus the composition morphism $I\overset{\Delta_I}{\longrightarrow}I\times_T I\cong G\times I$ maps $\sigma$ to $(\text{id}, \sigma)$.  Then gives rise to a polarized isomorphism $$\psi\colon (X_0\times I,L_{0}\times I) \cong (X\times_{T}I,L\times_{T}I).$$ For a point $\sigma\in I$, we claim that $\psi(x, \sigma)=(\sigma^{-1}(x), \sigma)$. In fact, restricting $\psi$ on the fiber over $t=\pi(\sigma)$, we have $$\sigma\in I_t=\mathrm{Isom}_{\{t\}}((X_t,L|_{X_t}),(X_0\times \{t\},L_{0} \times \{t\})),$$ that is, $\sigma$ gives rise to an isomorphism $\sigma\colon X_t\to X_0$. Hence $\psi(x,\sigma)=(\sigma^{-1}(x), \sigma)$.
        %We can verify that the action of $G$ on $X\times_{T}I$ \textcolor{red}{induced by the fiber product} is compatible with the diagonal action of $G$ on $X_0\times I$ as follows:
		%$$\xymatrixrowsep{0.4in}
        %\xymatrixcolsep{0.18in}
		%\xymatrix{&X_{0}\times I\ar[d]^{(g,g)}\ar[r]^{\psi}&X\times_{T}I\ar[d]^{(\mathrm{id}_{X}, g)}&(x,\sigma)\ar@{|_-{>}}[d]\ar@{|_-{>}}[r]&(\sigma^{-1} (x),\sigma)\ar@{|_-{>}}[d]&\\
			%&X_{0}\times I\ar[r]^{\psi}&X\times_{T}I&(g\cdot x, g\circ\sigma)\ar@{|_-{>}}[r]&\textcolor{red}{\bigg({\big((g\circ\sigma)}^{-1}\circ g\big)(x), g\circ\sigma\bigg)}=(\sigma^{-1} (x),g\circ\sigma).
        \medskip
		(3) By definition, $H$ acts naturally on $I_{0}\to T$. Since $I\to T$ is a $G$-torsor, $I_{0}\to T$ is an $H$-torsor. 
        \textcolor{red}{Hence $I$ is a disjoint union $\coprod_ig_iI_0$, and $G$ is a disjoint union $\coprod_ig_iH$. Each coset corresponds to a connected component of $I$, therefore $[G:H]<+\infty$. }%Fix a point $t\in T$, and $y\in I_{0,t}$, since $I\to T$ is a $G$-torsor, then $I_{t}=\{g\cdot y~|~g\in G\}$ and $I_{0,t}=\{h\cdot y~|~h\in H\}$. 
        
        %the remaining assertions in (3) follow immediately.
		
		\smallskip
		
		In particular, if $G=\mathrm{Aut}(X_{0},L_{0})$ is a finite group, then $I_{0}\to T$ is an \'etale $H$-cover and $X\cong X_{0}\times^{H}I_{0}$ by (2).

	\end{proof}

{\color{red}	\subsection{Albanese morphism and polarization of $K$-trivial canonical varieties} 
By a \textit {canonical variety} we mean a normal projective variety with at most canonical singularities. We say a canonical variety $X$ is \textit{$K$-trivial} if moreover the canonical divisor $K_{X}\sim_{\mathbb{Q}} 0$. For a $K$-trivial canonical  variety,  in \cite{kawamata1985minimal}*{Theorem 8.3} Kawamata deduced a decomposition theorem from the Albanese morphism, but the argument also applies for a morphism to an abelian variety. We have the following slight generalization.}
	%\begin{thm}{\cite{kawamata1985minimal}*{Theorem 8.3}}\label{thm:pol-isotr} Let $X$ be a normal projective variety with  at most canonical singularities, and let $\mathrm{alb}_{X}:X\to A=\mathrm{Alb}(X)$ be the Albanese morphism. Suppose that $K_{X}\sim_{\mathbb{Q}} 0$. Then $\mathrm{alb}_{X}$ is an \'etale fiber bundle, i.e., there is an isogeny $\tau \colon A'\to A$ of abelian varieties such that:$$X\times_{A}A'\cong F\times A'$$
		%where $F$ has at most canonical singularities and %$K_F\sim_{\mathbb{Q}} 0$.
        %\end{thm}
\textcolor{red}{\begin{thm}\label{thm:pol-isotr} Let $X$ be a $K$-trivial canonical variety. Let $h:X\to B$ be a morphism to an abelian variety $B$ such that the image $h(X)$ is nondegenerate in $B$. Let $X \xrightarrow{h'} B' \xrightarrow{\pi} B$ be the Stein factorization of $h:X\to B$. Then $\pi: B' \to B$ is an isogeny of abelian varieties, and $h':X\to B'$ is an \'etale fiber bundle, that is, there is an isogeny $\tau \colon B''\to B'$ of abelian varieties such that:
		$$X\times_{B'}B''\cong F\times B''$$
		where $F$ is a closed fiber of $h'$. In particular, $F$ is a $K$-trivial canonical variety.
\end{thm}}
\begin{proof} \textcolor{red}{Since $\pi\colon B' \to B$ is finite, it follows that $\kappa(B')\geq 0$ by \cite{kawamata1981characterization}*{Theorem 13}. By the universal property of the Albanese morphism $\mathrm{alb}_{X}:X\to \mathrm{Alb}(X)$,  there is a morphism $\eta \colon \mathrm{Alb}(X)\to B$ of abelian varieties such that $\eta\circ\mathrm{alb}_{X}=h$, and since $h(X)$ is nondegenerate in $B$, the morphism $\eta \colon \mathrm{Alb}(X)\to B$ is surjective. Then by \cite{kawamata1985minimal}*{Theorem 1.1}, we have $\mathrm{Var}(h')=\kappa(B')=0$. Hence by \cite{kawamata1981characterization}*{Theorem 4}, $\pi\colon B'\to B$ is an isogeny of abelian varieties. Finally, by applying exactly the same argument of \cite{kawamata1985minimal}*{Theorem 8.3}, we conclude that $h'\colon X\to B'$ is an \'etale fiber bundle. Since $X$ is a $K$-trivial canonical variety, so are $X\times_{B'}B''$ and $ F\times B''$, hence $F$ is a $K$-trivial canonical variety.}
\end{proof}

	Following Theorem \ref{thm:pol-isotr}, we see that the action of $H:=\mathrm{ker}(\tau)$ on $B''$ induces naturally an action on the base change $X\times_{B'} B'' \cong F\times B''$ such that $X\cong (F\times B'')/H$.  In fact $H$ acts on $F\times B''$ diagonally. \textcolor{red}{When $h:X\to B$ is the Albanese morphism, this result} has been proved in \cite{xu2020homogeneous} in log setting.
	But here to be self-contained and to maintain the information of polarization, we give an independent proof by the use of Isom functor. {\color{red} The key point is the following theorem, which implies that for the morphism $h':X\to B'$ in  Theorem \ref{thm:pol-isotr} and a fixed polarization $L$ on $X$, any two closed fibers of $h'$ with the restricted polarizations are polarized isomorphic to each other.}

	\begin{thm}\label{thm:ind-pol-auto}
		Let $X$ be a \textcolor{red}{$K$-trivial canonical variety}, and let $\mathrm{alb}_{X}: X\to A=\mathrm{Alb}(X)$ be the Albanese morphism. Then
		\begin{itemize}
			\item [(i)] For two ample line bundles $L_1,L_2$ on $X$, if $L_1\otimes L_2^{-1}\in \mathrm{Pic}^0(X)$, then there exists $\sigma\in \mathrm{Aut}^0(X)$ \footnote{$\mathrm{Aut}^0(X)$ is an abelian variety by Ueno \cite[Theorem 14.1]{book:Ueno1975}} such that $\sigma^*L_2 \sim L_1$.
			\item[(ii)] There is an isogeny $\tau \colon A'\to A$ of abelian varieties such that $X\cong A'\times^{H} F$ where $H=ker(\tau)$. Note that $F$ is a \textcolor{red}{$K$-trivial canonical variety}. Moreover we may choose the isogeny $\tau:A' \to A$ such that the action of $H$ on $F$ is faithful. 
		\end{itemize}
	\end{thm}
	
	\begin{proof}
    \textcolor{red}{Denote by $q(X)\colon=\mathrm{dim}H^{1}(X,\mathcal{O}_{X})$ the irregularity of $X$. If $q(X) = 0$, then the Albanese morphism is trivial, and both the two statements are trivial.}
	
		Now assume $q(X) >0$. We do induction on the dimension.
		Assume the statements hold for lower dimensional \textcolor{red}{$K$-trivial canonical varieties. We consider the following statements:}
		
		Pol($l$): Statetment (i) for $K$-trivial canonical varieties of dimension $l$.
		
		Diag($l$): Statement (ii) for $K$-trivial canonical varieties of dimension $l$.
		
		We follow the strategy
		$$\mathrm{Pol}(\leqslant n-1) \Rightarrow \mathrm{Diag}( n) \Rightarrow \mathrm{Pol}( n).$$
		\smallskip
		
		``$\mathrm{Pol}(\leqslant n-1) \Rightarrow \mathrm{Diag}(n)$'':
		\textcolor{red}{Assume that $\dim X=n$ and statement $\mathrm{Pol}(\leqslant n-1)$ holds.} When $q(X)=\dim(X)$, we have $X\cong \mathrm{Alb}(X)$ by \cite[Corollary 2]
		{kawamata1981characterization}, then the assertion (ii) is trivial. Now assume $0<q(X)<\dim(X)$. By Theorem \ref{thm:pol-isotr}, $\mathrm{alb}_{X}\colon X\to A$ is an analytic fiber bundle. Fix a polarization $L$ on $X$.
		\smallskip
		
		\textit{Claim}: For any two closed points $t_1, t_2\in A$, \textcolor{red}{there is a polarized isomorphism} $(X_{t_1}, L_{t_1}:=L|_{X_{t_1}}) \cong (X_{t_2}, L_{t_2}:=L|_{X_{t_2}})$.
		
		\textit{Proof of the Claim}: Since $\mathrm{alb}_{X}\colon X\to A$ is an analytic fiber bundle over a connected base, we can verify the statement locally. Fix a closed point $t_{0}\in A$. Take an analytic neighbourhood $U$ of $t_{0}$ in $A$ such that there is an isomorphism $\psi\colon U\times X_{t_{0}}\cong X_{U}$. Restricting $\psi$ to the fiber over a closed point $t\in U$ yields an isomorphism $\psi_{t}\colon X_{t_{0}}\cong X_{t}$. Set $L_{t_{0}}=L|_{X_{t_{0}}}$ and $L_{t}=L|_{X_{t}}$. Consider the holomorphic map $U\to \mathrm{Pic}^{0}(X_{t_{0}})$, $t\mapsto L_{t_{0}}\otimes \psi_{t}^{*}(L^{-1}_{t})$. Then for $t\in U$, $L_{t_{0}}\otimes \psi_{t}^{*}(L^{-1}_{t})\in \mathrm{Pic}^{0}(X_{t_{0}})$. \textcolor{red}{By Theorem \ref{thm:pol-isotr}, $X_{t_{0}}$ is a $K$-trivial canonical variety with $\mathrm{dim}(X_{t_{0}})\leqslant n-1$, then we can apply $\mathrm{Pol}(\leqslant n-1)$ to $X_{t_{0}}$, it follows that} there exists $\sigma\in \mathrm{Aut}^{0}(X_{t_{0}})$ such that $\sigma^{*}\psi^{*}L_{t}\cong L_{t_{0}}$. And we finish the proof of the claim.
		\smallskip
		
		\textcolor{red}{Since $X_{t_{0}}$ is a $K$-trivial canonical variety, ($X_{t_{0}},0)$ is a klt pair and $K_{X}$ is pseudo-effective, we have $\mathrm{Aut}(X_{t_{0}}, L_{t_{0}})$ is finite by Proposition \ref{finite-auto-pol}.} Applying Proposition \ref{prop:pol-isotr}, there exists a finite subgroup $H \leq \mathrm{Aut}(X_{t_{0}}, L_{t_{0}})$ and an \'etale $H$-cover $A'\to A$ such that $X\cong A'\times^{H}X_{t_{0}}$.

		\medskip
		``$\mathrm{Diag}(n) \Rightarrow \mathrm{Pol}(n)$'': \textcolor{red}{Assume that $\dim X=n$ and statement $\mathrm{Diag}(n)$ holds.} By assumption, $X\cong A'\times^{H}F$. Let $\pi\colon A'\times F\to X$ be the quotient morphism.
		
		First, we construct an $A'$-action on $X$. There is a natural action of $A'$ on $A'\times F$ by translations on the first factor. For each $a\in A'$, we denote by $t_{a}$ the translation by $a$. Then we can verify that this action commutes with the $H$-action on $A'\times F$ pointwisely:
		for $h\in H, a\in A'$ and $(a',f)\in A'\times F$,
		$$h\circ t_a (a',f)=(a'+a+h,h(f))=t_a\circ h(a',f),$$
		that is, the following diagram commutes:
		$$\xymatrix{&A'\times F\ar[d]^{t_a}\ar[r]^{h}&A'\times F\ar[d]^{t_a}&\\
			&A'\times F\ar[r]^{h}&A'\times F.
		}$$
		Therefore, the action of $t_a$ on $A'\times F$ descends to $\bar{t}_a \in \mathrm{Aut}^{0}(X)$ via the following commutative diagram:
		$$\xymatrix{&A'\times F\ar[d]^{t_a}\ar[r]^<(.15){\pi}&X\cong A'\times^{H}F\ar[d]^{\bar{t}_a}&\\
			&A'\times F\ar[r]^<(.15){\pi}&X\cong A'\times^{H}F.
		}$$
		\smallskip
		
		Next, we prove the following lemma, which is sufficient to deduce $\mathrm{Pol}(n)$.
		\smallskip
		
		\begin{lem}\label{diag-pol-trans}
			With the above notation, for an ample line bundle $L$ on $X$, the morphism
			$$\bar{\phi}_{L}\colon A'\rightarrow \mathrm{Pic}^0(X),~~a\mapsto \bar{t}_{a}^{*}L\otimes L^{-1}$$ is surjective.
		\end{lem}
		\begin{proof}
			Fix a closed point $f_{0}\in F$ and denote by $j:A'\times\{f_{0}\}\hookrightarrow A'\times F$ the closed immersion. Then $j^{*}\pi^{*}L$ is an ample line bundle on $A'\times\{f_{0}\}$. By Theorem \ref{pol-av}, the morphism
			$$\phi_{j^{*}\pi^{*}L}\colon A'\rightarrow \mathrm{Pic}^{0}(A'\times\{f_{0}\}),~~a\mapsto t^{*}_{a}(j^{*}\pi^{*}L)\otimes (j^{*}\pi^{*}L^{-1})$$
			is an isogeny of abelian varieties.
			Since $\pi\circ j:A'\times \{f_{0}\}\to X$ induces an isomorphism $H^{1}(X,\mathcal{O}_{X})\cong H^{1}(A'\times\{f_{0}\},\mathcal{O}_{A'\times\{f_{0}\}})$, the morphism $j^{*}\circ\pi^{*}:\mathrm{Pic}^{0}(X)\rightarrow\mathrm{Pic}^{0}(A'\times\{f_{0}\})$ is an isogeny of abelian varieties. By the construction of the $A'$-action on $X$, we can verify $\phi_{j^{*}\pi^{*}L}=j^{*}\circ\pi^{*}\circ\bar{\phi}_{L}$: for $a\in A'$,
			$$\phi_{j^{*}\pi^{*}L}(a)=t^{*}_{a}(j^{*}\pi^{*}L)\otimes (j^{*}\pi^{*}L^{-1})=j^{*}\pi^{*}(\bar{t}^{*}_{a}L)\otimes j^{*}\pi^{*}L^{-1}=j^{*}\circ\pi^{*}\circ\bar{\phi}_{L}(a).$$
			That is, the following diagram commutes:
			$$\xymatrixrowsep{0.4in}
			\xymatrixcolsep{0.5in}
			\xymatrix{
				A'\ar[dr]_{\phi_{j^{*}\pi^{*}L}}\ar[r]^{\bar{\phi}_{L}}&\mathrm{Pic}^{0}(X)\ar[d]^{j^{*}\circ \pi^{*}}\\
				&\mathrm{Pic}^{0}(A'\times \{f_{0}\}).
			}$$
			As a result, the morphism $\bar{\phi}_{L}:A'\to \mathrm{Pic}^{0}(X)$ is surjective.
		\end{proof}
		
		\smallskip
		
		Finally, we show $\mathrm{Pol}(n)$ by use of Lemma \ref{diag-pol-trans}. For two ample line bundles $L_{1}$, $L_{2}$ on $X$ such that $L_{1}\otimes L_{2}^{-1}\in \mathrm{Pic}^{0}(X)$, As the morphism $$\bar{\phi}_{L_{2}}:A'\to \mathrm{Pic}^{0}(X),a\mapsto \bar{t}_{a}^{*}L_{2}\otimes L_{2}^{-1}$$
		is surjective by Lemma \ref{diag-pol-trans}, there exists $a\in A'$ such that $\bar{t}_{a}^{*}L_{2
		}\otimes L_{2}^{-1}\cong L_{1}\otimes L_{2}^{-1}$, thus $\bar{t}_{a}^{*}L_{2
		}\cong L_{1}$.
		
	\end{proof}

{\color{red}Combining Proposition \ref{prop:pol-isotr} with Theorem \ref{thm:ind-pol-auto} we can show
    \begin{cor}
     With the same notation as in Theorem \ref{thm:pol-isotr} and setting $H:=\mathrm{ker}(\tau)$, we have $X\cong F\times^{H}B''$.
    \end{cor}
	 }

	\subsection{A special isotrivial fibration}
	We treat the following important special structure, which will play a key role in our proof. {\color{red} Though this result has already been proved in \cite[Proposition 3.6]{cch25nowhere}, we give an independent proof by use of Isom functor under slightly different conditions.}
	\begin{prop}\label{prop:isotrivial-prod}
		Let $X$ be a quasi-projective variety equipped with two morphisms
		$$\xymatrix{ &X \ar[r]^{g}\ar[d]^f & Z\\
			&A &
		}$$
		where $A$ is an abelian variety, \textcolor{red}{$
        Z$ is a smooth quasi-projective variety}, and
		$g\colon X\to Z$ is a smooth fibration. Fix $z_0\in Z$. Assume that
		\begin{itemize}
			\item[($\clubsuit$)] $X_{z_0} = A'$ is equipped with an abelian variety structure, such that the natural projection $f|_{A'}\colon A' \to A$ is an isogeny of abelian varieties.
			\end{itemize}
		Fix an ample line bundle $L_A$ on $A$, set $L = f^*L_{A}$ and $L_{A'}$ the restriction of $L$ on $A'=X_{z_0}$.
		
        Then \textcolor{red}{$g\colon X\to Z$ is a locally trivial fibration. And} there exist a finite subgroup $H \leq \mathrm{Aut}(A', L_{A'})$, and a variety $Z'$ equipped with a faithful action of $H$ such that
		\begin{itemize}
			\item[(i)] $Z \cong Z'/H$, and $X\cong A'\times^{H}Z'$;
			\item[(ii)] the projection $(A'\times Z')/H \to Z'/H \cong Z$ coincides with $g\colon X\to Z$; and
			\item[(iii)] $f\colon X\to A$ factors through the projection $(A'\times Z')/H \to A'/H$.
		\end{itemize}
		Remark that the assertion (iii) implies automatically that $H \leq ker(A'\to A)$.
		To summarize, we have the following commutative diagram:
		$$\xymatrix{ &  &A'\times Z'\ar[ld]_{p_1}\ar[d] \ar[r]^{\cong}&X\times_Z Z'\ar[d]^{q_{1}} \ar[r]^-{g'}  &Z'\ar[d] \\
			&A'\ar[d] &(A'\times Z')/H\ar[ld]\ar[d] \ar[r]^-{\cong}& X\ar[r]^>>>>{g}  \ar[ld]^{f} & Z\cong Z'/H\\
			&A'/H\ar[r]  &A &  &
		}$$
	\end{prop}
	
	\begin{proof} {\color{red}The local triviality of $g$ may be known to experts, but for lack of a suitable reference and for the convenience of the reader we include a proof. By \cite{Kodaira-spencer-I}*{Theorem 5.1}, to show that $g: X\to Z$ is locally trivial, we only need to verify that the Kodaira-Spencer map $\rho_z: T_{Z,z} \to H^1(X_z, T_{X_z})$ is zero at any closed point $z\in Z$. Firstly,
   since $X_{z_0}$ is isomorphic to an abelian variety, there is a Zariski dense open set $Z^{\circ} \subseteq Z$ such that $X^{\circ}:=g^{-1}(Z^{\circ})  \to Z^{\circ}$ is a family of abelian varieties. Secondly, by the results of N\'eron model \cite[Proposition 6]{neronmodel1990},  since $g:X \to Z$ is a smooth morphism, the whole family  $g:X \to Z$ is fibred by abelian varieties. Thirdly, the morphism
   $h=(f,g):X \to A\times Z$ of families over $Z$ induces a homomorphism of local systems
 $$h^*: H^1(A, \mathbb{Q})\times Z\to  R^1g_*\mathbb{Q}_X$$
 which corresponds to a homomorphism of variation of Hodge structures; as
the isogeny $X_{z_0} \to A\times \{z_0\}$ induces $H^1(A, \mathbb{Q}) \times \{z_0\} \cong (R^1g_*\mathbb{Q}_X)_{z_0}$, we have that $h^*: H^1(A, \mathbb{Q}) \times Z \cong  R^1g_*\mathbb{Q}_X$. Therefore, the variation of Hodge structure associated to $R^1g_*\mathbb{Q}_X$ is constant, in particular, the differential of the period map 
$$d\mathcal{P}^{1,1}: T_{Z,z} \xrightarrow{\rho_z} H^1(X_z, T_{X_z}) \xrightarrow{\mu} \mathrm{Hom}_{\mathbb{C}}(H^0(X_z, \Omega_{X_z}^1), H^1(X_z, \mathcal{O}_{X_z}))$$
is zero. Finally, by the local Torelli theorem for abelian varieties \cite[Theorem 10.27]{book:hodgeI}, $\mu: H^1(X_z, T_{X_z}) \to \mathrm{Hom}_{\mathbb{C}}(H^0(X_z, \Omega_{X_z}^1), H^1(X_z, \mathcal{O}_{X_z}))$ is an isomorphism, we conclude that the Kodaira-Spencer map $\rho_z: T_{Z,z} \to H^1(X_z, T_{X_z})$ is zero at any closed point $z\in Z$.}

\textcolor{red}{Moreover, for a closed point $z\in Z$, we have an isomorphism $(f|_{X_{z}})^{*}:H^1(A, \mathbb{Q}) \times \{z\} \cong  (R^1g_*\mathbb{Q}_X)_{z}\cong H^1(X_{z}, \mathbb{Q}_{X_{z}})$, it follows that $f|_{X_{z}}:X_{z}\to A$ is a finite \'etale cover, therefore $L|_{X_{z}}$ is an ample line bundle on $X_{z}$. As a result $L$ is $g$-ample by \cite[Corollary 9.6.4]{EGAIV}.} By Theorem \ref{thm:ind-pol-auto}, we have $(X_{z},L|_{X_{z}})\cong (A',L_{A'})$ for each closed point $z\in Z$, \textcolor{red}{and by Proposition \ref{finite-auto-pol}, $\mathrm{Aut}(A',L_{A'})$ is finite.} Applying Proposition \ref{prop:pol-isotr}, there exists a finite subgroup $H \leq \mathrm{Aut}(A', L_{A'})$ and an $H$-cover $Z'\to Z$ satisfying (i) and (ii). \textcolor{red}{By Proposition \ref{prop:pol-isotr} (2), we obtain} a polarized isomorphism over $Z'$
$$\beta\colon (A'\times Z',L_{A'}\times Z')\xrightarrow{\sim}(X\times_{Z}Z',q_{1}^{*}L).$$
Consider the composition of polarized morphisms
$$\psi \colon (A'\times Z',L_{A'}\times Z')\xrightarrow[\cong]{\beta}(X\times_{Z}Z',q_{1}^{*}L)\xrightarrow{q_{1}}(X,L)\xrightarrow{f}(A,L_{A}).$$
Then $\psi$ induces a morphism: $$\phi\colon Z'\to \mathrm{Mor}((A',L_{A'}),(A,L_{A})).$$
Since $\mathrm{Mor}((A',L_{A'}),(A,L_{A}))$ is a finite scheme by Lemma \ref{lem:pol-mor-finite} and $Z'$ is connected, we conclude that $\phi$ is a constant morphism,   \color{red} in other words, for any closed point $z' \in Z'$, if identifying $A' = A'\times \{z'\}$ then the morphism $\psi|_{A'\times \{z'\}}: A' \cong A'\times \{z'\} \to A$ is independent of  $z'$. In turn, we conclude that the morphism $A'\times Z'\to A$ factors through the projection $p_{1}\colon A'\times Z'\to A'$ to the first factor, and we can obtain the factorization claimed in the statement (iii).
\end{proof}

\section{Proof of The Main Results}
      \begin{proof}[Proof of Main Theorem \ref{thm:structure}] 
      \textcolor{red}{Note that if $\dim S=0$ then Theorem \ref{thm:structure} is a direct consequence of Theorem \ref{thm:pol-isotr}. Therefore, in what follows we assume that $\dim S >0$. Since $X$ is normal with at most canonical singularities, then by \cite[Lemma 5.17]{KM98bir}, there exists a nonempty open subset $S^{\circ}\subseteq S$ such that for any closed point $s\in S^{\circ}$, the fiber of $f$ over $s$, denoted by $X_{s}$, is a normal projective variety with at most canonical singularities. Hence, $X_{s}$ is a $K$-trivial canonical variety.}
      \smallskip

      \textcolor{red}{To prove (1),} let $X\xrightarrow{g}A\xrightarrow{\tau}B$ be the Stein factorization of $h\colon X\to B$.         
		Note that even though $f\colon X\to S$ and $g\colon X\to A$ are fibrations, the product morphism $f\times g\colon X\to S\times A$ does not necessarily have connected fibers. Let $X\to Y\to S\times A$ be the Stein factorization of $f\times g\colon X\to S\times A$. \textcolor{red}{Set $X^{\circ}:=X\times_{S}S^{\circ}$ and $Y^{\circ}:=Y\times_{S}S^{\circ}$. Then $X^{\circ}\to Y^{\circ}\to S^{\circ}\times A$ is the Stein factorization of $X^{\circ}\to S^{\circ}\times A$, and for each closed point $s\in S^{\circ}$,  $X_{ s}\to Y_{s}\to \{s\}\times A$ is the Stein factorization of $g\big|_{X_{s}}\colon X_{s}\to A$. Since $X_{s}$ is a $K$-trivial canonical variety and the composition $X_{s}\to \{s\}\times A \to  \{s\}\times B$ is surjective, by Theorem \ref{thm:pol-isotr}, $Y_{s}\to \{s\}\times A \to \{s\}\times B$ is a finite \'etale cover, hence both $Y_s$ and $A$ are isomorphic to abelian varieties isogenous to $B$, and we have}
		\begin{itemize}
			\item [($\sharp$)] $X_{s}\to Y_{s}$ is an \textcolor{red}{\'etale fiber bundle.}
            \end{itemize}
\textcolor{red}{Fix a closed point $s_{0}\in S^{\circ}$ and let $A_{1}=Y_{s_{0}}$. We can equip $A_{1}$ and $A$ with abelian variety structures such that $A_{1}\to A$ is an isogeny of abelian varieties.}
		\smallskip
		
		\textcolor{red}{To prove (2), }we shall do a sequence of base changes to attain an isogeny $A'\to A$ such that $X\times_A A'$ splits. We break the arguments into four steps.
		
		\medskip
		
 		\par\textbf{Step 1}. In this step, we \textcolor{red}{show that the base change of $X^{\circ}\to Y^{\circ}$ along $A_{1}\to A$ factors through a connected component $Y_{1}^{\circ}\cong S_{1}^{\circ}\times A_{1}$ of $Y^{\circ}\times_{A}A_{1}$, and that the composition morphism $X_{1}^{\circ}:=X^{\circ}\times_{A}A_{1}\to Y_{1}^{\circ}\cong S_{1}^{\circ}\times A_{1}$ is a fibration coinciding with the product morphism $f_{1}\times g_{1}:X_{1}^{\circ}\to S_{1}^{\circ}\times A_{1}$, where $f_{1}:X_{1}^{\circ}\to S_{1}^{\circ}$ and $g_{1}:X_{1}^{\circ}\to A_{1}$ are two base changes of $f:X^{\circ}\to S^{\circ}$ and $g:X^{\circ}\to A$ respectively.} And we build the following \textcolor{red}{commutative} diagram:
		$$\xymatrixcolsep{0.6in}
		\xymatrix{&X_{1}^{\circ}\ar[d]\ar[r]^-{f_{1}\times g_{1}}&S^{\circ}_{1}\times A_{1}\ar[d]\ar[r]&S^{\circ}\times A_{1}\ar[d]\\
			&X^{\circ}\ar[r]&Y^{\circ}\ar[r]&S^{\circ}\times A\\
		}$$
		\smallskip
		\\
		\textit{Proof of Step 1}: {\color{red}After possibly replacing $S^{\circ}$ by a nonempty open subset of its regular locus, we can assume that $Y^{\circ}\to S^{\circ}$ is flat, hence  $Y^{\circ}\to S^{\circ}$ is a smooth fibration by \cite[Theorem 10.2]{har1977}.}  Then by Proposition \ref{prop:isotrivial-prod}, \textcolor{red}{$Y^{\circ}\to S^{\circ}$ is a locally trivial fibration, and there exist a subgroup $H_{1}\leq\mathrm{ker}(A_{1}\to A)$ and an $H_{1}$-cover $S_{1}^{\circ}\to S^{\circ}$ such that $Y^{\circ}\cong S^{\circ}_{1}\times ^{H_{1}}A_{1}$.} Meanwhile, we have \textcolor{red}{$Y_{1}^{\circ}:=Y^{\circ}\times_{S^{\circ}}S_{1}^{\circ}\cong S^{\circ}_{1}\times A_{1}\cong Y^{\circ}\times_{A_{1}/H_{1}}A_{1}$, which is a connected component of $ Y^{\circ}\times_{A}A_{1}$.} 
		More precisely, we have the following commutative diagram
		$$\xymatrixcolsep{0.5in}
		\xymatrixrowsep{0.5in}
		\xymatrix{&&&X_{1}^{\circ}=X^{\circ}\times_{A}A_{1}\ar[d]^{g_{1}}\ar@/^1.5pc/[ldd]^{f_{1}}\ar[dr]\ar[dl]_{\gamma}&\\
			&&\textcolor{red}{Y_{1}^{\circ}:=Y^{\circ}\times_{S^{\circ}}S_{1}^{\circ}\cong S^{\circ}_{1}\times A_{1}}\ar[dr]\ar[r]\ar[d]&A_{1}\ar[dr]&X^{\circ}\ar[d]^{g}\ar[dl]\ar@/^1.5pc/[ldd]^{f}&\\
			&&S_{1}^{\circ}\ar[dr]&Y^{\circ}\ar[r]\ar[d]&A\\
			&&&S^{\circ}&\\
		}$$
		where $f_{1}:X_{1}^{\circ}\to S_{1}^{\circ}$ and $g_{1}:X_{1}^{\circ}\to A_{1}$ are two base changes of $f:X^{\circ}\to S^{\circ}$ and $g:X^{\circ}\to A$ respectively, and $\gamma: X_{1}^{\circ}\to Y_{1}^{\circ}$ is the base change of $X^{\circ}\to Y^{\circ}$.
		
		\smallskip
		
		Under the isomorphism \textcolor{red}{$Y_{1}^{\circ}:=Y^{\circ}\times_{S^{\circ}}S_{1}^{\circ}\cong S^{\circ}_{1}\times A_{1}$}, the projections from $Y_{1}^{\circ}$ to $S_{1}^{\circ}$ and $A_1$ coincide with  the projections from the product $S^{\circ}_{1}\times A_{1}$ to its two factors. Hence, by the universal property of the fiber product, the morphism $X_{1}^{\circ}\xrightarrow{\gamma} Y_{1}^{\circ}\xrightarrow{\sim}S_{1}^{\circ}\times A_{1}$ coincides with the product morphism $f_{1}\times g_{1}:X_{1}^{\circ}\to S_{1}^{\circ}\times A_{1}$. In particular, since $\gamma: X_{1}^{\circ} \to Y_{1}^{\circ}$ is a fibration, so is $f_{1}\times g_{1}\colon X^{\circ}_{1}\to S^{\circ}_{1}\times A_{1}$. 
		\medskip
		\begin{notation}
			Let $L$ be an $f_{1}\times g_{1}$-ample line bundle on $X^{\circ}_{1}$. \textcolor{red}{Fix a closed point $a_{0}\in A_{1}$, denote by $X^{\circ}_{a_{0}}:=g_{1}^{-1}(a_{0})$ the fiber of $g_{1}\colon X_{1}^{\circ}\to A_{1}$ over $a_{0}$.}
            Let $$I=\mathrm{Isom}_{S^{\circ}_{1}\times A_{1}}((X^{\circ}_{1},L),(X^{\circ}_{a_{0}}\times A_{1},L|_{X^{\circ}_{a_{0}}} \times A_{1})).$$
			\textcolor{red}{By Proposition \ref{prop:isom-rep}(i), I is a quasi-projective scheme over $S^{\circ}_{1}\times A_{1}$. The fiber of $I\to {S_{1}}^{\circ}$ over a closed point $s\in {S_{1}}^{\circ}$ is
			$$I_{s}=\mathrm{Isom}_{\{s\}\times A_{1}}(((X^{\circ}_{1})_{s},L|_{(X^{\circ}_{1})_{s}}),((X^{\circ}_{a_{0},s}\times A_{1},L|_{X^{\circ}_{a_{0},s}} \times A_{1})).$$  From  the condition ($\sharp$) and by Proposition \ref{prop:pol-isotr}, it follows that $I_{s}\to A_{1}$ is finite \'etale and components of $I_{s}$ are isomorphic to an abelian variety. After possibly replacing $S^{\circ}$ by an non empty open subset, we can assume that $I\to {S_{1}}^{\circ}$ is a projective smooth morphism.  Let $I\to {S'_{1}}^{\circ} \to S^{\circ}_{1}$ be the Stein factorization of $I\to S^{\circ}_{1}$. Then by Proposition \ref{prop:isotrivial-prod}, $I\to {S'_{1}}^{\circ}$ is a locally trivial family of abelian varieties over ${S'_{1}}^{\circ}$.}
		\end{notation}
		\medskip
		\par \textbf{Step 2}. In this step, we aim to find an isogeny of abelian varieties $A_{2}\to A_{1}$ and a Galois cover $S_{2}^{\circ}\to S_{1}^{\circ}$, so that there is a polarized isomorphism over $S_{2}^{\circ}\times A_{2}$ as follows: $$\beta\colon((X^{\circ}_{1}\times_{A_{1}}A_{2})\times_{S_{1}^{\circ}}S_{2}^{\circ},(L\times_{A_{1}}A_{2})\times_{S_{1}^{\circ}}S_{2}^{\circ})
		\xrightarrow{\sim}
		((X^{\circ}_{a_{0}}\times_{S^{\circ}_{1}}S^{\circ}_{2}),(L|_{X^{\circ}_{a_{0}}}\times_{S^{\circ}_{1}}S^{\circ}_{2}))\times A_{2},$$
		that is, we obtain the following commutative diagram:
		$$\xymatrix{(X^{\circ}_{a_{0}}\times_{S^{\circ}_{1}}S^{\circ}_{2})\times A_{2}&
			(X^{\circ}_{1}\times_{A_{1}}A_{2})\times_{S_{1}^{\circ}}S_{2}^{\circ} \ar[r]\ar[dd]\ar[l]_{\beta}
			^{\cong}& X^{\circ}_{1}\times_{A_{1}}A_{2}\ar[r]\ar[d] & A_{2}\ar[d] \\
			&& X^{\circ}_{1}\ar[r]^{g_{1}}\ar[d]^{f_{1}} & A_{1} \\
			&S^{\circ}_{2}\ar[r] & S^{\circ}_{1}
		}$$
		\textit{Proof of Step 2}: %Recall that $I\to {S'_{1}}^{\circ}$ is a locally trivial family of abelian varieties. 
        Fix a closed point $s'_{0}\in {S'_{1}}^{\circ}$ and let $A_{2}=I_{s'_{0}}$. \textcolor{red}{As $I_{s'_{0}}\to A_{1}$ is a finite \'etale cover, we can} equip $A_{2}$ with an abelian variety structure such that $A_{2}\to A_{1}$ is an isogeny. Fix an ample line bundle \textcolor{red}{$L_{A_{1}}$ on $A_{1}$, and let $L_{A_{2}}$ be its pullback to $A_{2}$.} Applying Proposition \ref{prop:isotrivial-prod}, there exists a subgroup $H_{2}\leq \mathrm{Aut}(A_{2},L_{A_{2}})$ and an $H_{2}$-cover $S^{\circ}_{2}\to {S'_{1}}^{\circ}$ such that $I\cong S^{\circ}_{2}\times^{H_{2}}A_{2}$. By the construction, we have the following commutative diagram:
		$$\xymatrix{ &  &S_{2}^{\circ}\times A_{2}\ar[ld]_{pr_2}\ar[d] \ar[r]^{\cong}&I\times_{{S'_{1}}^{\circ}} S_{2}^{\circ}\ar[d]\ar[r] &S_{2}^{\circ}\ar[d] \\
			&A_{2}\ar[d] &(S_{2}^{\circ}\times A_{2})/H_{2}\ar[ld]\ar[d] \ar[r]^>>>>>>>>{\cong}& I\ar[r]\ar[ld]\ar[rd] & {S'_{1}}^{\circ}\cong S_{2}^{\circ}/H_{2}\ar[d]\\
			&A'_{1}=A_{2}/H_{2}\ar[r]  &A_{1} &  &S_{1}^{\circ}~~~.
		}$$
		Note that the morphism $I\to A_{1}$ factors through $I\to A'_{1}=A_{2}/H_{2}$ and each closed fiber of $I\to A'_{1}$ is isomorphic to $S_{2}^{\circ}$. From this we conclude that $S_{2}^{\circ}\times A_{2}\cong I\times_{A_{1}'}A_{2}$, which is an isomorphism over $S^{\circ}_{1}\times A_{2}$. Since $I\times_{A'_{1}}A_{2}$ is a connected component of $I\times_{A_{1}}A_{2}$, the composition morphism
		$$\alpha:S^{\circ}_{2}\times A_{2}\cong I\times_{A_{1}'}A_{2}\hookrightarrow I\times_{A_{1}}A_{2}$$
		is a morphism over $S^{\circ}_{1}\times A_{2}$.
		\smallskip
		\\By Proposition \ref{prop:isom-rep} (i), the scheme
		$I\times_{A_{1}}A_{2}$ represents the Isom functor $$\mathrm{\textbf{Isom}}_{S^{\circ}_{1}\times A_{2}}((X_1^{\circ}\times_{A_{1}}A_{2},L\times_{A_{1}}A_{2}),(X^{\circ}_{a_{0}}\times A_{2},L|_{X^{\circ}_{a_{0}}}\times A_{2}))(-).$$
		The morphism $\alpha \colon S^{\circ}_{2}\times A_{2}\to I\times_{A_{1}}A_{2}$ gives rise to an element
		$$\beta\in\mathrm{\textbf{Isom}}_{S^{\circ}_{1}\times A_{2}}((X^{\circ}_{1}\times_{A_{1}}A_{2},L\times_{A_{1}}A_{2}),(X^{\circ}_{a_{0}}\times A_{2},L|_{X^{\circ}_{a_{0}}}\times A_{2}))(S^{\circ}_{2}\times A_{2}),$$
		that is, $\beta$ is a polarized isomorphism over $S_{2}^{\circ}\times A_{2}$ as follows:
		$$\beta\colon((X^{\circ}_{1}\times_{A_{1}}A_{2})\times_{S_{1}^{\circ}}S_{2}^{\circ},(L\times_{A_{1}}A_{2})\times_{S_{1}^{\circ}}S_{2}^{\circ})
		\xrightarrow{\sim}
		((X^{\circ}_{a_{0}}\times_{S^{\circ}_{1}}S^{\circ}_{2}),(L|_{X^{\circ}_{a_{0}}}\times_{S^{\circ}_{1}}S^{\circ}_{2}))\times A_{2}.$$
		
		\smallskip
		
		In addition, if we take a further base change $\tilde{S_{2}^{\circ}}\to S_{2}^{\circ}$ over $S_{1}^{\circ}$, then the composition morphism $$\tilde{\alpha}:\tilde{S_{2}^{\circ}}\times A_{2}\to S^{\circ}_{2}\times A_{2}\xrightarrow{\alpha} I\times_{A_{1}}A_{2}$$ is also a morphism over  $S_{1}^{\circ}\times A_{2}$, in turn we also have a polarized isomorphism over $\tilde{S_{2}^{\circ}}\times A_{2}$. So we can replace $S^{\circ}_{2}$ by its Galois closure over $S^{\circ}_{1}$ to assume that $S^{\circ}_{2}\to S^{\circ}_{1}$ is a Galois cover.
		
		\medskip
		
		\textbf{Step 3}. In this step, we aim to show $X^{\circ}_{1}\times_{A_{1}}A_{2}\cong X^{\circ}_{a_{0}}\times A_{2}$.\medskip\\
		\textit{Proof of Step 3}: Let $G:=\mathrm{Gal}(S^{\circ}_{2}/S^{\circ}_{1})$.
		Then there is an action of $G$ on $(X^{\circ}_{1}\times_{A_{1}}A_{2})\times_{S_{1}^{\circ}}S_{2}^{\circ}$ by base changes, which preserves the polarization $(L\times_{A_{1}}A_{2})\times_{S_{1}^{\circ}}S_{2}^{\circ}$. Since the morphism $f_{1}\times g_{1}:X_{1}^{\circ}\to S_{1}^{\circ}\times A_{1}$ is a fibration, the action of $G$ preserves the fibers of the projection $(X^{\circ}_{1}\times_{A_{1}}A_{2})\times_{S_{1}^{\circ}}S_{2}^{\circ} ~{\color{red}\cong (X^{\circ}_{1}\times_{S_{1}^{\circ}}S_{2}^{\circ})\times_{A_{1}}A_{2} }\to A_{2}$.
		More precisely, for $g\in G$, we have the following commutative diagram:
		{\footnotesize $$\xymatrixcolsep{0.05in}
			\xymatrix{\textcolor{red}{\bigg(}(X^{\circ}_{1}\times_{A_{1}}A_{2})\times_{S_{1}^{\circ}}S_{2}^{\circ},(L\times_{A_{1}}A_{2})\times_{S_{1}^{\circ}}S_{2}^{\circ}\textcolor{red}{\bigg)}\ar[rr]^{g}_{\cong}\ar[d]&&\textcolor{red}{\bigg(}(X^{\circ}_{1}\times_{A_{1}}A_{2})\times_{S_{1}^{\circ}}S_{2}^{\circ},(L\times_{A_{1}}A_{2})\times_{S_{1}^{\circ}}S_{2}^{\circ}\textcolor{red}{\bigg)}\ar[d]\\
				S_{2}^{\circ}\times A_{2}\ar[rr]^{(g,\mathrm{id})}_{\cong}\ar[dr]&&S_{2}^{\circ}\times A_{2}\ar[dl]\\
				&S_{1}^{\circ}\times A_{2}&}$$}
		Remind that in Step 2 we have constructed a polarized isomorphism over $S_{2}^{\circ}\times A_{2}$:
		$$\beta\colon((X^{\circ}_{1}\times_{A_{1}}A_{2})\times_{S_{1}^{\circ}}S_{2}^{\circ},(L\times_{A_{1}}A_{2})\times_{S_{1}^{\circ}}S_{2}^{\circ})
		\xrightarrow{\sim}
		((X^{\circ}_{a_{0}}\times_{S^{\circ}_{1}}S^{\circ}_{2}),(L|_{X^{\circ}_{a_{0}}}\times_{S^{\circ}_{1}}S^{\circ}_{2}))\times A_{2}.$$
		In turn, we obtain an action of $G$ on $(X^{\circ}_{a_{0}}\times_{S^{\circ}_{1}}S^{\circ}_{2})\times A_{2}$ such that $\beta$ is $G$-equivariant:
		{\footnotesize \begin{align*}
			G\times ((X^{\circ}_{a_{0}}\times_{S^{\circ}_{1}}S^{\circ}_{2})\times A_{2},(L|_{X^{\circ}_{a_{0}}}\times_{S^{\circ}_{1}}S^{\circ}_{2})\times A_{2})&\to((X^{\circ}_{a_{0}}\times_{S^{\circ}_{1}}S^{\circ}_{2})\times A_{2},(L|_{X^{\circ}_{a_{0}}}\times_{S^{\circ}_{1}}S^{\circ}_{2})\times A_{2}),\notag\\
			(g,(x,a_{2}))&\mapsto(g_{a_{2}}(x),a_{2})\notag
		\end{align*}}
	here $g_{a_{2}}$ is a polarized automorphism of $((X^{\circ}_{a_{0}}\times_{S^{\circ}_{1}}S^{\circ}_{2})\times \{a_{2}\},(L|_{X^{\circ}_{a_{0}}}\times_{S^{\circ}_{1}}S^{\circ}_{2})\times \{a_{2}\})$ over $S_{1}^{\circ}$. This action induces a morphism $$\eta:A_{2}\to \mathrm{Hom}_{gp}\bigg(G,\mathrm{Aut}_{S_{1}^{\circ}}(X^{\circ}_{a_{0}}\times_{S^{\circ}_{1}}S^{\circ}_{2},L|_{X^{\circ}_{a_{0}}}\times_{S^{\circ}_{1}}S^{\circ}_{2})\bigg).$$
		Since $|G|<\infty$ and $\mathrm{Aut}_{S_{1}^{\circ}}(X^{\circ}_{a_{0}}\times_{S^{\circ}_{1}}S^{\circ}_{2},L|_{X^{\circ}_{a_{0}}}\times_{S^{\circ}_{1}}S^{\circ}_{2})$ is a group scheme over \textcolor{red}{$\mathbb{C}$}, we see that $\eta$ is constant, that is, the action of $G$ on $(X^{\circ}_{a_{0}}\times_{S_{1}^{\circ}}S_{2}^{\circ})\times A_{2}$ is independent of $a_{2}\in A_{2}$. From this we conclude that $X_{1}^{\circ}\times_{\textcolor{red}{A_{1}}}A_{2}\cong ((X^{\circ}_{a_{0}}\times_{S_{1}^{\circ}}S_{2}^{\circ})\times A_{2})/G\cong X^{\circ}_{a_{0}}\times A_{2}$.
		
		\medskip
		
		\par \textbf{Step 4}. Let $X'=X\times_{A}A_{2}$ which is \'etale over $X$, then $X'$ is again $\Q$-factorial and has canonical singularities, and $K_{X'}$ is semi-ample. By (the paragraph after) \cite[Corollary 1.4.3]{birkar2010existence}, we have the terminalization morphism $\eta\colon \widetilde{X}\to X'$ such that $\eta$ is a crepant morphism, i.e., $K_{\widetilde{X}}=\eta^*K_{X'}$, $\widetilde{X}$ is terminal and $\Q$-factorial. Also, since $K_{X'}$ is semi-ample, so is $K_{\widetilde{X}}$. Then $\widetilde{X}$ is a good minimal model of $X'$.
		
		By Step 3, we have a birational map $\widetilde{X}\stackrel{bir}\dashrightarrow X^{\circ}_{a_{0}}\times A_{2}$. Take a projectivization and desingularization $Y$ of $X^{\circ}_{a_{0}}$. Then $Y\times A_{2}$ has \textcolor{red}{a} good minimal model, \textcolor{red}{which is} $\widetilde{X}$. By \cite[Proposition 2.5]{Lai2011varsfiberedbygoodminimalmodel}, any directional MMP of $Y\times A_{2}$ teminates. Also, since $A_2$ is an abelian variety, each step (extremal contractions and flips) of the MMP of $Y\times A_2$ reduces to a MMP step of $Y$. Therefore, we get a minimal model $\widetilde{Y}$ of $Y$ such that $\widetilde{Y}\times A_2$ is another minimal model of $X'$. By \cite[Theorem 1]{kawamata2008flops}, the birational map $\widetilde{X}\stackrel{bir}\dashrightarrow \widetilde{Y}\times A_{2}$ is a composition of flops. Again since $A_2$ is an abelian variety, the flops starting from $\widetilde{Y}\times A_{2}$ arise from flops of $\widetilde{Y}$. Therefore, we get $\widetilde{X}\cong Z\times A_2$ with $Z$ a minimal model of $X_{a_{0}}^\circ$. Hence we have the following commutative diagram
		
		$$\xymatrixrowsep{0.6in}
		\xymatrix{
			\widetilde{X}\ar[d]_{\eta}\ar[r]^-{\cong}&Z\times A_2\ar[d]^{p_2}\\
			X'\ar[r]&A_{2}
		}$$
		Since $K_{\widetilde{X}}=\eta^*K_{X'}$, $\eta$ is the contraction of a $K_{\widetilde{X}}$-trivial face. By the above commutative diagram, for any curve $C\subset \widetilde{X}\cong Z\times A_2$ contracted by $\eta$, $p_2(C)$ is a point. Hence the union of the curves in the numerical classes of the aforementioned $K_{\widetilde{X}}$-trivial face form a trivial fiber bundle over $A_2$. Hence the contraction $\eta$ arises from a contraction of $Z$. Therefore, $X'\cong Z'\times A_2$. Since the projection $X'\to A_2$ comes from the \'etale base change along $A_2\to A$, we get that the morphism $g\colon X\to A$ is a fiber bundle.
		
	\end{proof}
	\medskip
	
	\begin{proof}[Proof of Theorem \ref{thm:main2}]
		
		The direction $(ii)\Rightarrow (i)$ is trivial. We will show $(i)\Rightarrow (ii)$. Consider the Albanese morphism $\mathrm{alb}_{X}$ and Iitaka fibration $f$ of $X$:
		$$\xymatrixcolsep{0.5in}
		\xymatrixrowsep{0.5in}
		\xymatrix{
			&X\ar[d]_{f}\ar[r]^{\mathrm{alb}_{X}}&A_{X}&\\
			&S& &
		}$$
		Since $X$ admits a holomorphic 1-form $\omega$ without zeros,
		then by \cite[Theorem 2.1]{popa2014kodaira}, $\mathrm{alb}_{X}$ maps a general fiber $F$ of $f$ to a translation of some fixed positive dimensional abelian subvariety $B_{0}$ of $A_{X}$, {\color{red}that is}, $\mathrm{alb}_{X}(F)=B_{0}+t_{F}\subset A_{X}$, for some $t_{F}\in A_{X}$ depending on $F$. Taking the dual of the injection $B_{0}\rightarrow A_{X}$, we have the following commutative diagram:
		$$\xymatrixrowsep{0.5in}
		\xymatrixcolsep{0.5in}
		\xymatrix{
			&F\ar[r]^{\mathrm{alb}_{X}|_{F}}\ar@{^(-_>}[d]&B_{0}+t_{F}\ar@{^(-_>}[d]\ar[rrd]^{\gamma}&\\
			&X\ar@/_2pc/[rrr]_{h}\ar[r]^{\mathrm{alb}_{X}}&A_{X}\ar[r]&A_{X}^{\vee}\ar[r]&B_{0}^{\vee}
		}$$
		where $\gamma$ is a finite \'etale cover. \textcolor{red}{And by Theorem \ref{thm:pol-isotr}, $h\big|_F$ is surjective}.
		Let $X\xrightarrow{g}A\xrightarrow{\sigma}B^{\vee}_{0}$ be the Stein factorization of $h$. Then by Theorem \ref{thm:structure}, $A$ is an abelian variety and $g\colon X\to A$ is a fiber bundle.
		
		Now we show $\dim A=\dim B_0\geq g$. Since $\mathrm{alb}_{X}$ maps the fibers of $f$ onto translations of $B_0$ and $S$ has rational singularties, we have the following commutative diagram:
		$$\xymatrixrowsep{0.5in}
		\xymatrixcolsep{0.5in}
		\xymatrix{
			X\ar[d]_f\ar[r]^{\mathrm{alb}_{X}}&A_{X}\ar[d]^p\\
			S\ar[r]^{\eta}&A_X/B_0~.&
		}$$
		Then again by \cite[Theorem 2.1]{popa2014kodaira}, for any holomorphic 1-form $\omega$ on $A_X/B_0$, $\mathrm{alb}_X^*p^*\omega$ has zero. If $\dim B_0<g$, then $\dim A_X/B_0> \dim A_X-g$, which contradicts to the assumption that $X$ admits  $g$ pointwise linearly independent holomorphic 1-forms.
	\end{proof}

\begin{proof}[Proof of Theorem \ref{thm:main3}]
    \textcolor{red}{Let $X^{\mathrm{min}}$ be a good minimal model of $X$,
and let $\tau:X\dashrightarrow X^{\mathrm{min}}$ be a birational map. Denote by ${a_{X}}$ and $a_{X^{\mathrm{min}}}$ the Albanese morphisms of $X$ and $X^{\mathrm{min}}$ respectively. Let $\phi:X^{\mathrm{min}}\to S$ be the Iitaka fibration of $X^{\mathrm{min}}$.}
\smallskip

\textcolor{red}{\textit{Claim:} For a general fiber $F$ of $\phi$, $\dim a_{X^{\mathrm{min}}}(F)>0$.}
\smallskip

\textcolor{red}{\textit{Proof of the Claim:} Assume, for the sake of contradiction, that $\dim a_{X^{\mathrm{min}}}(F)=0$ for a general fiber $F$ of $\phi$. Then $\dim a_{X^{\mathrm{min}}}(F)=0$ for every fiber $F$ of $\phi$ by the rigidity lemma, see \cite[Lemma 1.6]{KM98bir}. It follows that $a_{X^{\mathrm{min}}}$ facts through $\phi$, that is, there exists a morphism $\beta:S\to A$ such that $a_{X^{\mathrm{min}}}=\beta\circ\phi$. We obtain the following commutative diagram
$$\xymatrixcolsep{0.8in}\xymatrixrowsep{0.5in}
    \xymatrix{X\ar@{.>}[d]^{\tau}\ar[rd]^{{a_{X}}}&\\
    X^{\mathrm{min}}\ar[d]^{\phi}\ar[r]^{a_{X^{\mathrm{min}}}}&A.\\
    S\ar[ru]_{\beta}&}$$} 
{\color{red} But then \cite[Theorem 2.1]{popa2014kodaira} tells that every holomorphic 1-form $\omega\in H^{0}(X,\Omega_{X}^{1})$ has zeros on $X$, a contradiction. This proves the claim.}
    \medskip
    
       \textcolor{red}{Now, for a general fiber $F$ of $\phi$, the image $a_{X^{\mathrm{min}}}(F)$ is a translation of a subtorus $B$ of $A$. As in the proof of Theorem \ref{thm:main2}, we obtain the following commutative diagram
$$\xymatrix{X\ar@{>>}[r]\ar@{.>}[d]_(.4){\tau}&B^{\vee}~.&\\X^{\mathrm{min}}\ar@{>>}[ur]&}$$
Let $X^{\mathrm{min}}\to B'\to B^{\vee}$ be the Stein factorization of $X^{\mathrm{min}}\to B^{\vee}$. By Theorem \ref{thm:structure}, $B'$ is an abelian variety, and there exists an isogeny $B''\to B'$ of abelian varieties such that $X^{\mathrm{min}}\times_{B'}B''\cong Z\times B''$, where $Z$ is a fiber of $X^{\mathrm{min}}\to B'$. Hence $X\times_{B'}B''\dashrightarrow B''\times Z$ is birational.}
\end{proof}

	\bibliographystyle{amsalpha}

\end{document}